\documentclass[a4paper,11pt,reqno]{amsart}

\usepackage{fullpage}

\usepackage{amsmath,amssymb,dsfont}

\usepackage{hyperref} 

\usepackage[nobysame,alphabetic,initials,msc-links]{amsrefs}

\DefineSimpleKey{bib}{how}

\renewcommand{\eprint}[1]{#1}
\BibSpec{misc}{%
  +{}{\PrintAuthors}  {author}
  +{,}{ \textit}      {title}
  +{,}{ }             {how}
  +{}{ \parenthesize} {date}
  +{,} { available at \eprint}        {eprint}
  +{,}{ available at \url}{url}
  +{,}{ }             {note}
  +{.}{}              {transition}
}

\usepackage[OT2,T1]{fontenc}
\DeclareSymbolFont{cyrletters}{OT2}{wncyr}{m}{n}
\DeclareMathSymbol{\Sha}{\mathalpha}{cyrletters}{"58}

\numberwithin{equation}{section}

\newtheorem{theorem}{Theorem}[section]
\newtheorem{corollary}[theorem]{Corollary}

\newtheorem{lemma}[theorem]{Lemma}
\newtheorem{proposition}[theorem]{Proposition}

\theoremstyle{remark}
\newtheorem{remark}[theorem]{Remark}
\newtheorem{example}[theorem]{Example}
\theoremstyle{definition}
\newtheorem{definition}[theorem]{Definition}

\theoremstyle{plain}
\newtheorem{innercustomthm}{Theorem}
\newenvironment{customthm}[1]
  {\renewcommand\theinnercustomthm{#1}\innercustomthm}
  {\endinnercustomthm}

\newcommand\bp{\begin{proof}}
\newcommand\ep{\end{proof}}
\newcommand\ee{\nopagebreak\mbox{\ }\hfill$\diamond$}

\newcommand{\un}{\mathds{1}}

\makeatletter
\newcommand{\llangle}{\mathopen{\langle\mkern-4.5mu\langle}}
\newcommand{\rrangle}{\mathclose{\rangle\mkern-4.5mu\rangle}}
\makeatother

\newcommand\Ad{\operatorname{Ad}}
\newcommand\ad{\operatorname{ad}}

\newcommand\Aut{\operatorname{Aut}}

\newcommand\Hom{\operatorname{Hom}}

\newcommand\Op{\operatorname{Op}}
\newcommand\Tr{\operatorname{Tr}}
\newcommand\tr{\operatorname{tr}}
\newcommand{\BCH}{\operatorname{BCH}}

\newcommand{\C}{{\mathbb C}}
\newcommand{\E}{\mathbb E}
\newcommand{\FF}{{\mathbb F}}

\newcommand{\Q}{{\mathbb Q}}

\newcommand\T{{\mathbb T}}
\newcommand\Z{{\mathbb Z}}

\newcommand{\F}{{\mathcal F}}

\newcommand\OO{\mathcal O}

\newcommand{\g}{\mathfrak g}
\newcommand{\n}{\mathfrak n}
\newcommand{\p}{\mathfrak p}

\newcommand\blhd{\blacktriangleleft}

\begin{document}

\date{October 6, 2026; revised October 21, 2026}

\title{Facets of nondegeneracy: central type factor groups, Heisenberg representations and Hadamard matrices}

\author[Bieliavsky]{Pierre Bieliavsky}
\address{Institut de Recherche en Math\'ematique et Physique, Universit\'e Catholique de Louvain, Chemin du Cyclotron, 2, 1348 Louvain-la-Neuve, Belgium}
\email{Pierre.Bieliavsky@uclouvain.be}

\author[Gayral]{Victor Gayral}
\address{Laboratoire de Math\'ematiques, CNRS UMR 9008, Universit\'e de Reims Champagne-Ardenne,
Moulin de la Housse - BP 1039,
51687 Reims, France}
\email{victor.gayral@univ-reims.fr}

\author[Neshveyev]{Sergey Neshveyev}
\address{Department of Mathematics, University of Oslo,
P.O. Box 1053 Blindern, NO-0316 Oslo, Norway}
\email{sergeyn@math.uio.no}

\author[Tuset]{Lars Tuset}
\address{Department of Computer Science, OsloMet - storbyuniversitetet,
P.O. Box 4 St. Olavs plass, NO-0130 Oslo, Norway}
\email{larst@oslomet.no}

\begin{abstract}
Given a pair of complementary subgroups $P$ and $N$ of a finite group $G$, we first discuss how the same type of functions $\mathbb E\colon P\times N\to \mathbb T$, which we call bi-$1$-cocycles, appear naturally in the study of projective representations of $G$ of Heisenberg type and second cohomology of $G$ and of its dual quantum group $\widehat G$. We show then that such representations are irreducible and the corresponding ordinary and dual $2$-cocycles are nondegenerate if and only if the matrix $(\mathbb E(p,n))_{p,n}$ is invertible, if and only if it is Hadamard. We illustrate the result with examples arising from symplectic nilpotent Lie algebras over finite fields with two complementary Lagrangian subalgebras.
\end{abstract}

\maketitle

\bigskip

\section*{Introduction}

This paper is motivated by the problem of quantization of groups, but its principal result can be formulated as a solution to a very natural problem in representation theory. Assume we are given a finite group $G$ with two complementary subgroups $P$ and $N$, meaning that $P\cap N=\{e\}$ and $|P|\,|N|=|G|$, in which case one writes $G=P\bowtie N$. By a representation of $G$ of Heisenberg type we mean any projective unitary representation $\pi\colon G\to U(H)$ of dimension $|G|^{1/2}$ such that its restrictions to~$P$ and~$N$ are honest representations. The problem is to decide when such a representation is irreducible.

In order to formulate the result, note that by assumption any element of $G$ can be written uniquely as $pn$, with $p\in P$ and $n\in N$. The same element can also be written as a product of elements of $N$ and $P$, so we have an exchange relation
$$
pn=(p\rhd n)(p\lhd n),
$$
with $p\rhd n\in N$ and $p\lhd n\in P$. Applying $\pi$ we get
$$
\Phi(p,n)\pi(p)\pi(n)=\pi(p\rhd n)\pi(p\lhd n)
$$
for some phase factor $\Phi(p,n)\in\T$. Here is one possible formulation of our main result.

\begin{customthm}{A}\label{thm:A}
Assume $G=P\bowtie N$ and $\pi\colon G\to U(H)$ is a projective unitary representation of~$G$ of Heisenberg type with respect to this factorization, with the corresponding phase factor~$\Phi$. Then the following conditions are equivalent:
\begin{enumerate}
  \item the representation $\pi$ is irreducible;
  \item the matrix $(\Phi(p,n))_{p\in P,n\in N}$ is invertible;
  \item the matrix $(\Phi(p,n))_{p\in P,n\in N}$ is Hadamard.
\end{enumerate}
Moreover, if these conditions are satisfied, then $\pi$ is determined uniquely, up to projective unitary equivalence (and even up to unitary equivalence if we in addition require $\pi(pn)=\pi(p)\pi(n)$), by~$\Phi$. We then call $\pi$ the Heisenberg representation of $G$ with phase~$\Phi$.
\end{customthm}

Recall that a (complex) Hadamard matrix is an $m\times m$ matrix $H=(H_{ij})_{i,j}$ such that $|H_{ij}|=1$ for all $i,j$ and $H$ is a scalar multiple of a unitary matrix, that is, $HH^*=mI$. As we will see, the Hadamard matrices in Theorem~\ref{thm:A} are of Butson type, that is, their entries are roots of unity.

The phase $\Phi$ defines a $\T$-valued $2$-cocycle $\sigma$ on $G$, and irreducibility of $\pi$ is equivalent to simplicity of the twisted group algebra $\C^\sigma G$. Cocycles with this property are called nondegenerate and groups admitting such cocycles are known as central type factor groups (ctfgs). In principle, it is well-known how to characterize nondegeneracy without talking about representations or twisted group algebras: the requirement is that for every $g\ne e$ the character $h\mapsto\sigma(g,h)\overline{\sigma(h,g)}$ on the centralizer $C(g)$ of $g$ should be nontrivial. Although this characterization is explicit, applying it may require substantial information about the centralizers in $G$. Theorem~\ref{thm:A} gives an alternative criterion for cocycles admitting two complementary isotropic subgroups.

If we replace a representation $\pi$ of Heisenberg type by another one in the same projective equivalence class, this in general transforms the phase $\Phi$ in a nontrivial way. Namely, we can multiply $\pi|_P$ and $\pi|_N$ by some characters $\chi_P$ and $\chi_N$, and then the  new phase is
$$
\Phi(p,n)\frac{\chi_P(p\lhd n)}{\chi_P(p)}\frac{\chi_N(p\rhd n)}{\chi_N(n)}.
$$
Already examples coming from semidirect products of abelian groups show that such transformations can dramatically change the rank of $(\Phi(p,n))_{p,n}$. Theorem~\ref{thm:A} implies that, somewhat surprisingly, once we have reached the maximal possible rank $|G|^{1/2}$, the same transformations can neither change the rank nor destroy the Hadamard property.

Every nondegenerate $2$-cocycle on $G$ is known to determine a cohomology class of $2$-cocycles on the dual quantum group $\widehat G$~\cites{MR733301,MR990110,MR1264314,MR1747109}, that is, invertible elements $\Omega\in\C G\otimes\C G$ such that
$$
(\Omega\otimes1)(\Delta\otimes\iota)(\Omega)=(1\otimes\Omega)(\iota\otimes\Delta)(\Omega).
$$
The cocycles $\Omega$ defined by nondegenerate cocycles on $G$ are themselves nondegenerate in the same sense as before -- the corresponding twisted quantum group algebras are simple. A special feature of the cocycles $\Omega$ that we get from Theorem~\ref{thm:A} is that they can be chosen so that their support in the group basis is exactly $P\times N$ or $N\times P$. It turns out that, up to a normalization, all cocycles on $\widehat G$ with this property are obtained this way.

\begin{customthm}{B}\label{thm:B}
Assume $G=P\bowtie N$. Then, for any function $\Phi\colon P\times N\to\C^\times$ with $\Phi(e,e)=1$, the following conditions are equivalent:
\begin{enumerate}
  \item the element $\Omega:=|G|^{-1/2}\sum_{p\in P,n\in N}\Phi(p,n^{-1})\,p\otimes n\in\C G\otimes\C G$ is a $2$-cocycle on $\widehat G$;
  \item the function $\Phi$ is the phase of an irreducible projective unitary representation of Heisenberg type, so that it satisfies the equivalent conditions of Theorem~\ref{thm:A}.
\end{enumerate}
Moreover, if these conditions are satisfied, then the cocycle $\Omega$ is nondegenerate and unitary.
\end{customthm}


The present work grew out of our attempts to find analytical quantizations of certain Lie groups, for which we need verifiable conditions that guarantee unitarity of cocycle candidates on the dual quantum group, cf.~\cites{GM,BGNT4,BGNT5}. This more general setting is still in the process of development and will be discussed in detail elsewhere.

In the finite group case one can also look at the above results not as a way to construct cocycles on $\widehat G$, but as a curious source of examples of Hadamard matrices arising from ctfgs. Hadamard matrices have long been viewed  as some sort of quantum Fourier matrices. For the Hadamard matrices $(\Phi(p,n))_{p,n}$ this point of view is particularly justified, as the operators $\ell^2(N)\to\ell^2(P)$ defined by them have a number of similarities with the usual Fourier transform  for abelian groups and play a central role in our analysis.

One of the best-known constructions of ctfgs is the bijective $1$-cocycle method developed by Etingof and Gelaki~\cites{MR1653340,MR1747109}, see also~\cites{MR2534251,BGNT4} for generalizations. In its semidirect-product form, one starts with a group $G=P\ltimes N$, where $N$ is abelian, such that the dual action of~$P$ on~$\widehat N$ can be cocycle perturbed to a free transitive affine action $P\curvearrowright\widehat N$. These examples fit into our framework, but their phase matrices are the Fourier matrices of~$N$, up to permutations of rows or columns.

To obtain examples beyond this construction, we consider symplectic nilpotent Lie algebras $(\mathfrak g,\omega)$ over finite fields with two complementary Lagrangian subalgebras. When the characteristic is sufficiently large, the Lazard correspondence produces a finite group with  a nondegenerate group cocycle and two complementary isotropic subgroups. Our criterion then yields an Hadamard matrix, which truly deserves to be called a quantization of a Fourier matrix defined by~$\omega$. In this nilpotent setting, the Hadamard property can also be deduced from results in~\cite{MR2801175} on unitary intertwiners between representations defined using different polarization models.

We carry out the construction explicitly for an eight-dimensional symplectic Lie algebra having complementary Lagrangian subalgebras but no Lagrangian ideals. The corresponding nondegenerate group cocycle lies outside the bijective $1$-cocycle construction and its extension to groups with normal Lagrangian subgroups~\cites{MR2534251,MR3210715}. The resulting Hadamard matrices are neither equivalent to Fourier matrices nor of Di\c{t}\u{a} type. We do not claim that they cannot be obtained by other known constructions, but this example shows in any case that  the Hadamard matrices in Theorem~\ref{thm:A} can be quite nontrivial.

Let us finally say that connections between Hadamard matrices, group cocycles and monomial representations have been studied extensively. Classical cocyclic constructions associate to a cocycle on a group $G$ a matrix of size $|G|$, see, for example,~\cite{MR2265694}, rather than of size $|G|^{1/2}$ as in our case. More general approaches use intertwining spaces of monomial representations, with representation-theoretic conditions yielding orthogonality, see, in particular, \cites{MR4815669,ACDE}. Our construction belongs to this broader framework.

\medskip

\paragraph{\bf Use of AI tools.}
The final section of the paper was written with extensive help of OpenAI's ChatGPT 6.0 Astra. In particular, AI found a relevant eight-dimensional example in the literature, adapted it to the finite field setup and performed all computations. AI was also used to check arguments and to suggest references and improvements to the exposition throughout the paper. The authors verified the resulting arguments and computations and take full responsibility for the content.

\bigskip

\section{Matched pairs and cocycles}

Throughout the paper $G$ denotes a finite group. Given subgroups $P,N\subset G$, one says that $(P,N)$ is a \emph{matched pair} if the map $P\times N\to G$, $(p,n)\mapsto pn$, is a bijection. This is equivalent to requiring $P\cap N=\{e\}$ and $|P|\,|N|=|G|$. We then write $G=P\bowtie N$. We get a left action $\rhd\colon P\times N\to N$ and a right action $\lhd\colon P\times N\to P$ defined by the identity
$$
pn=(p\rhd n)(p\lhd n).
$$
The maps $p\rhd$ and $\lhd n$ are called \emph{dressing transformations}. They are not group automorphisms in general, but they still satisfy certain cocycle-type relations that we list below.

\begin{lemma}
The following identities hold:
\begin{equation}\label{eq:dress1}
(p_1p_2)\lhd n=\big(p_1\lhd(p_2\rhd n)\big)(p_2\lhd n),\qquad p\rhd(n_1n_2)=(p\rhd n_1)\big((p\lhd n_1)\rhd n_2\big).
\end{equation}
In particular, the dressing transformations act trivially on the unit element and we have
\begin{equation}\label{eq:dress2}
(p^{-1}\lhd n)^{-1}=p\lhd(p^{-1}\rhd n),\qquad (p\rhd n^{-1})^{-1}=(p\lhd n^{-1})\rhd n.
\end{equation}
\end{lemma}

It will sometimes be convenient to conjugate the right action of $N$ on $P$ by the inverse, that is, we define
a new right action $\blhd\colon P\times N\to P$ by
$$
p\blhd n:=(p^{-1}\lhd n)^{-1}.
$$
The point is that this puts $P$ and $N$ on a more equal footing, since $p\blhd n=n^{-1}\rhd p$, where the last action is defined by the matched pair $(N,P)$.

\smallskip

We will study $2$-cocycles on $G$ of a particular form. Recall first that a $\T$-valued $2$-cocycle $\sigma\in Z^2(G;\T)$ is a map $\sigma\colon G\times G\to\T$ satisfying the identity
$$
\sigma(g,h)\sigma(gh,k)=\sigma(g,hk)\sigma(h,k)\qquad (g,h,k\in G).
$$
The coboundaries are cocycles of the form $(\partial b)(g,h)=b(g)b(h)b(gh)^{-1}$, where $b\colon G\to\T$.

It was shown by Mackey~\cite{MR360923} that if $G=P\bowtie N$, then up to coboundaries any cocycle decomposes into three terms, two of which arise from cocycles on $P$ and $N$, while the third one is defined by a function $P\times N\to\T$. When the first two terms vanish, this leads to the following definition.

\begin{definition}\label{def:b1-cocycle}
We say that a function $\E\colon P\times N\to\T$ is a \emph{bi-$1$-cocycle} on $G=P\bowtie N$ if it satisfies the following identities:
\begin{equation}\label{eq:b1-cocycle1}
\E(p_1p_2,n)=\E(p_1,n)\E(p_2,p_1^{-1}\rhd n),\qquad \E(p,n_1n_2)=\E(p,n_1)\E((p^{-1}\lhd n_1)^{-1},n_2).
\end{equation}
\end{definition}

As we will see, in terms of the phase functions $\Phi$ used in the introduction, the bi-$1$-cocycles are the functions $\E(p,n)=\Phi(p^{-1},n)$.

Note that in terms of the action $\blhd$ the second relation for $\E$ becomes
\begin{equation*}\label{eq:b1-cocycle2}
\E(p,n_1n_2)=\E(p,n_1)\E(p\blhd n_1,n_2).
\end{equation*}
Observe also that by taking $n_i=e$ and $p_i=e$ in the above identities we get
\begin{equation*}\label{eq:normalization}
\E(p,e)=\E(e,n)=1.
\end{equation*}

The definition of a bi-$1$-cocycle makes of course sense for any abelian target group instead of~$\T$. In particular, one might be tempted to consider $\C^\times$-valued bi-$1$-cocycles, but as we will see shortly, by finiteness of $G$ they automatically take values in the roots of unity.

\smallskip

Given a bi-$1$-cocycle $\E\colon P\times N\to\T$, it can be checked (and in an indirect way this will be done in the next section by constructing certain projective representations) that we get a $2$-cocycle~$\sigma_\E$ on~$G$ defined by
\begin{equation}\label{eq:MackeyE}
\sigma_{\E}(p_1n_1,p_2n_2):=\E(p_2,n_1^{-1}).
\end{equation}

Given a class $c\in H^2(G,\T)$, a subgroup $H\subset G$ is called \emph{isotropic} with respect to $c$ if this cohomology class has trivial restriction to $H$, that is, $\operatorname{res}^G_Hc=0$ in $H^2(H;\T)$. It is clear that the subgroups $P$ and $N$ are isotropic with respect to the cohomology classes $[\sigma_\E]$. The following proposition, which is essentially a particular case of results of Mackey~\cite{MR360923} or Kac~\cite{Kac} (see Remark~\ref{rem:Kac} below), shows that this is a characteristic property of these classes.

\begin{proposition}\label{prop:Mackey}
Assume $G=P\bowtie N$ and $c\in H^2(G;\T)$ is a class with respect to which~$P$ and~$N$ are isotropic. Then $c$ can be represented by a cocycle $\sigma\in Z^2(G;\T)$ such that $\sigma|_{P\times P}=1$ and $\sigma|_{N\times N}=1$. For any such $\sigma$, the function $\E\colon P\times N\to\T$ defined by
\begin{equation}\label{eq:b1-from-2}
\E(p,n):=\sigma(p^{-1}\rhd n, p^{-1}\lhd n)\sigma(p^{-1},n)^{-1}
\end{equation}
is a bi-$1$-cocycle and $c=[\sigma_\E]$, where $\sigma_\E$ is the $2$-cocycle given by~\eqref{eq:MackeyE}. All bi-$1$-cocycles $\E$ are obtained this way from appropriate $c$ and $\sigma$.
\end{proposition}

\bp
Consider the central extension $1\to\T\to G_c\to G\to1$ corresponding to $c$. Fix a section $s\colon G\to G_c$. The assumption that $P$ and $N$ are isotropic implies that we can modify $s$ on $P$ and $N$ so that it becomes a homomorphism on these groups. Then the corresponding cocycle $\sigma(g,h):=s(g)s(h)s(gh)^{-1}$ has the properties $\sigma|_{P\times P}=1$ and $\sigma|_{N\times N}=1$.

Assume now that we are given a cocycle $\sigma$ representing the class $c$ such that $\sigma|_{P\times P}=1$ and $\sigma|_{N\times N}=1$. We can find a section $s\colon G\to G_c$ defining it. Then $s|_P$ and $s|_N$ are homomorphisms. Applying $s$ to the identity $pn=(p\rhd n)(p\lhd n)$ we get
$$
\sigma(p,n)^{-1}s(p)s(n)=\sigma(p\rhd n,p\lhd n)^{-1}s(p\rhd n)s(p\lhd n),
$$
which shows that $\E$ is determined by the identity
\begin{equation}\label{eq:exchange}
\E(p,n)s(p^{-1})s(n)=s(p^{-1}\rhd n)s(p^{-1}\lhd n).
\end{equation}
Applying this separately to $p=p_1p_2$ and $n=n_1n_2$ we get the bi-$1$-cocycle identities.

Next, define a new section $t\colon G\to G_c$ by $t(pn):=s(p)s(n)$. Then
\begin{align}
t(p_1n_1)t(p_2n_2)&=s(p_1)\big(s(p_2^{-1})s(n_1^{-1})\big)^{-1}s(n_2)\notag\\
&=\E(p_2,n_1^{-1})s(p_1)s\big((p_2^{-1}\lhd n_1^{-1})^{-1}\big)s\big((p_2^{-1}\rhd n_1^{-1})^{-1}\big)s(n_2)\notag\\
&=\E(p_2,n_1^{-1})t(p_1n_1p_2n_2).\label{eq:cocycle-computation}
\end{align}
Thus, the cocycle defined by the section $t$ is $\sigma_\E$, so $c=[\sigma_\E]$.

To prove the last part of the proposition it suffices to show that the construction~\eqref{eq:b1-from-2}, being applied to $\sigma_\E$, gives back $\E$, that is, it defines a left inverse to the map
$$
\{\text{bi-$1$-cocycles}\}\to\{\text{$2$-cocycles that are identically $1$ on $P$ and $N$}\},\quad \E\mapsto\sigma_\E.
$$
In order to show this, we need to check that for any bi-$1$-cocycle $\E$ we have the identity
\begin{equation}\label{eq:flip}
\E(p^{-1}\lhd n,(p^{-1}\rhd n)^{-1})=\E(p,n).
\end{equation}
Replacing $p$ by $(p\lhd n^{-1})^{-1}$ and using~\eqref{eq:dress2}, we equivalently need to check that
$$
\E(p,p\rhd n^{-1})=\E(p^{-1}\blhd n^{-1},n).
$$
This identity is true, since by the cocycle relations both sides are equal to $\E(p^{-1},n^{-1})^{-1}$.
\ep

\begin{corollary}\label{cor:cobound}
Given two bi-$1$-cocycles $\E_1$ and $\E_2$, we have $[\sigma_{\E_1}]=[\sigma_{\E_2}]$ in $H^2(G;\T)$ if and only if there exist homomorphisms
$\chi_P\colon P\to\T$ and $\chi_N\colon N\to\T$ such that
$$
\E_1(p,n)=\E_2(p,n)\frac{\chi_P(p^{-1})}{\chi_P(p^{-1}\lhd n)}\frac{\chi_N(n)}{\chi_N(p^{-1}\rhd n)}.
$$
\end{corollary}

\bp
Consider the class $c:=[\sigma_{\E_1}]$, the corresponding extension $1\to\T\to G_c\to G\to1$ and a section $t\colon G\to G_c$ defining $\sigma_{\E_1}$. We then have to understand what kind of bi-$1$-cocycles we can get by modifying $t$ in such a way that the corresponding $2$-cocycle is identically $1$ on~$P$ and~$N$, that is, the new section $s$ must still be a homomorphism on $P$ and $N$. From~\eqref{eq:exchange} we see, however, that for the construction of a bi-$1$-cocycle we only need to know the restrictions of~$s$ to~$P$ and~$N$. Therefore the only meaningful choice for $s$ that we can make is to multiply~$t|_P$ and~$t|_N$ by some characters $\chi_P$ and $\chi_N$. This gives the result.
\ep

\begin{remark}\label{rem:Kac}
The bi-$1$-cocycles form a group under multiplication. It is isomorphic (or equal, depending on the conventions) to the first Kac cohomology group $H^1_{\mathrm{Kac}}(P, N;\T)$, see~\cites{Kac,MR2115374}; we are indebted to Arthur Massar for this remark. Kac's cohomology fits into a long exact sequence, and the above considerations reprove exactness of the following part of it:
\[
H^1(P;\T)\oplus H^1(N;\T)\to H^1_{\mathrm{Kac}}(P, N;\T)\to H^2(G;\T)\to H^2(P;\T)\oplus H^2(N;\T).\tag*{$\diamond$}
\]
\end{remark}

\begin{corollary}\label{cor:roots}
Any $\C^\times$-valued bi-$1$-cocycle $\E$ takes values in the group $\mu_m$ of roots of unity of order $m:=|G|\operatorname{lcm}(|P|,|N|)$.
\end{corollary}

\bp
As is well-known, any cohomology class in $H^2(G;\C^\times)$ has order dividing $|G|$, see~\cite{MR672956}*{Corollary III.10.2}. Therefore $[\sigma_{\E^{|G|}}]=[\sigma_\E^{|G|}]=0$ in~$H^2(G;\C^\times)$. By Corollary~\ref{cor:cobound}, or rather by its proof for $\C^\times$ instead of $\T$, applied to $\E^{|G|}$ and the trivial bi-$1$-cocycle, it follows that there exist homomorphisms $\chi_P\colon P\to\C^\times$ and $\chi_N\colon N\to\C^\times$ such that
$$
\E(p,n)^{|G|}=\frac{\chi_P(p^{-1})}{\chi_P(p^{-1}\lhd n)}\frac{\chi_N(n)}{\chi_N(p^{-1}\rhd n)}.
$$
Since $\chi_P$ and $\chi_N$ take values in the groups $\mu_{|P|}$ and $\mu_{|N|}$, we get the result.
\ep


We will show next that the notion of a bi-$1$-cocycle appears naturally also in the study of dual $2$-cocycles. Consider the group algebra $\C G$ of $G$. It is a C$^*$-algebra with involution $g^*:=g^{-1}$. Denote by $\Delta\colon\C G\to\C G\otimes\C G$ the comultiplication, $\Delta(g):=g\otimes g$.

An invertible element $\Omega\in\C G\otimes\C G$ is called a $2$-cocycle on the dual quantum group $\widehat G$, or a \emph{dual $2$-cocycle} on $G$, if
\begin{equation}\label{eq:dual-2-cocycle}
(\Omega\otimes 1)(\Delta\otimes\iota)(\Omega)=(1\otimes\Omega)(\iota\otimes\Delta)(\Omega).
\end{equation}
We remark that the name \emph{twist} is common in the algebraic literature, and then it is often not~$\Omega$ but $\Omega^{-1}$ that is called a twist.

\begin{proposition}
\label{prop:dual-bi1}
Assume $G=P\bowtie N$ and $\E\colon P\times N\to\C^\times$ is a function. Consider the element
\begin{equation*}
\label{eq:dual-crossed-term}
\Omega:=\sum_{p\in P,n\in N}\E(p^{-1},n^{-1})\,p\otimes n\in\C G\otimes\C G.
\end{equation*}
Then $\Omega$ satisfies the dual cocycle identity~\eqref{eq:dual-2-cocycle} if and only if $\E(e,e)^{-1}\E$ is a bi-$1$-cocycle.
\end{proposition}

We stress that $\Omega$ is not invertible without further restrictions on $\E$, and so it is not a dual $2$-cocycle in general.

\begin{proof}
Since~\eqref{eq:dual-2-cocycle} is stable under multiplication of $\Omega$ by a scalar, we may assume that $\E(e,e)=1$. Since $\E$ is a bi-$1$-cocycle if and only if $\bar\E$ is, we can also replace $\E$ by $\bar\E$ in the proof, so that the adjoint element $\Omega^*$ becomes equal to $\sum_{p,n}\E(p,n)\,p\otimes n$. We need to show that $\E$ is a bi-$1$-cocycle if and only if
\begin{equation}\label{eq:dual-2-cocycle2}
(\Delta\otimes\iota)(\Omega^*)(\Omega^*\otimes 1)=(\iota\otimes\Delta)(\Omega^*)(1\otimes\Omega^*),
\end{equation}
that is,
$$
\sum_{p,p',n,n'} \E(p,n)\E(p^{-1}p',n')\,p'\otimes pn'\otimes n=
\sum_{p,p',n,n'} \E(p',n')\E(p,n'^{-1}n)\,p'\otimes n'p\otimes n.
$$
If we write $n'p$ as $\tilde p\tilde n$, then $n'=\tilde p\rhd\tilde n$ and $p=\tilde p\lhd\tilde n$. Hence the right hand side of the above equation can be written as
$$
\sum_{p,p',n,n'} \E(p',p\rhd n' )\E(p\lhd n',(p\rhd n')^{-1}n\big)\, p'\otimes pn'\otimes n.
$$
Therefore $\Omega^*$ satisfies~\eqref{eq:dual-2-cocycle2} if and only if $\E$ satisfies
$$
\E(p,n)\E(p^{-1}p',n')=\E(p',p\rhd n' )\E(p\lhd n',(p\rhd n')^{-1}n\big),
$$
or equivalently, in a more symmetric form,
\begin{equation}
\label{ECond}
\E(p,n'n)\E(p',p^{-1} \rhd n')=\E(pp',n')\E((p^{-1}\lhd n')^{-1},n),
\end{equation}
where we have used that $p\lhd(p^{-1}\rhd n')=(p^{-1}\lhd n')^{-1}$ by~\eqref{eq:dress2}.

Now, if $\E$ is a bi-$1$-cocycle, then the cocycle relations~\eqref{eq:b1-cocycle1} imply that both sides of~\eqref{ECond} equal
$$
\E(p,n')\E((p^{-1}\lhd n')^{-1},n)\E(p',p^{-1} \rhd n').
$$
Conversely, assume~\eqref{ECond} is satisfied.
Evaluating it at  $p'=e$ and $n'=e$ gives
$\E(p,e)=1$, since $\E(e,e)=1$ by assumption, and evaluating at $p=e$ and $n=e$ gives
$\E(e,n')=1$.
The two bi-$1$-cocycle relations follow then by evaluating \eqref{ECond} at  $n=e$ for the first one and at  $p'=e$ for the second.
\end{proof}

\bigskip

\section{Facets of nondegeneracy}

We continue to assume that $G$ is a finite group. By a projective unitary representation of~$G$ on a finite-dimensional Hilbert space $H$ one means either a homomorphism $G\to PU(H)$ or a map $\pi\colon G\to U(H)$ that is a homomorphism modulo scalar operators. For us it will be convenient to adopt the second point of view. By a projective intertwiner between two projective representations $\pi_1\colon G\to U(H_1)$ and $\pi_2\colon G\to U(H_2)$ one means then a linear operator $T\colon H_1\to H_2$ such that $T\pi_1(g)=c(g)\pi_2(g)T$ for some scalars $c(g)\in\C^\times$. If there exists a unitary projective intertwiner, then $\pi_1$ and $\pi_2$ are called projectively unitarily equivalent.

Given a projective unitary representation $\pi\colon G\to U(H)$, the corresponding cocycle $\sigma\in Z^2(G;\T)$ is defined by the identity
$$
\pi(g)\pi(h)=\sigma(g,h)\pi(gh).
$$
Then one says also that $\pi$ is a $\sigma$-representation. Once $\sigma$ is fixed, it makes sense to talk about ordinary intertwiners, when we require $T\pi_1(g)=\pi_2(g)T$,  and unitary equivalences between $\sigma$-representations. Two $\sigma$-representations $\pi_1$ and $\pi_2$ are projectively unitarily equivalent if and only if $\pi_1$
and $\chi\pi_2$ are unitarily equivalent for some $\chi\in\Hom(G,\T)$.

The twisted group algebra $\C^\sigma G$ is defined as the space $\C G$ equipped with the product $g\cdot_\sigma h:=\sigma(g,h)gh$. It is a finite-dimensional C$^*$-algebra with unit $\overline{\sigma(e,e)}e$ and involution
$$
g^{*_\sigma}:=\overline{\sigma(g^{-1},g)\sigma(e,e)}g^{-1}.
$$
We have an obvious one-to-one correspondence between unitary $\sigma$-representations of $G$ and unital $*$-representations of $\C^\sigma G$. As $\C^\sigma G$ is semisimple, it follows then that a nonzero $\sigma$-representation $\pi\colon G\to U(H)$ is irreducible if and only if the corresponding representation $\C^\sigma G\to B(H)$ is surjective.

A cocycle $\sigma\in Z^2(G;\T)$ is called \emph{nondegenerate} if the algebra $\C^\sigma G$ is simple, equivalently, it is isomorphic  as a C$^*$-algebra to the algebra of $|G|^{1/2}\times|G|^{1/2}$ matrices; in particular, $|G|$ must be a square. Up to an isomorphism, the algebra $\C^\sigma G$ depends only on the cohomology class of $\sigma$, and as a consequence nondegeneracy is a property of the class $[\sigma]\in H^2(G;\T)$. Since full matrix algebras have a very simple representation theory, one can formulate the following characterization of nondegeneracy.

\begin{lemma}\label{lem:Nondeg}
For any $\sigma\in Z^2(G;\T)$, the following conditions are equivalent:
\begin{enumerate}
  \item the cocycle $\sigma$ is nondegenerate;
  \item there is an irreducible unitary $\sigma$-representation $\pi$ of dimension $|G|^{1/2}$;
  \item the smallest nonzero dimension for which there exists a unitary $\sigma$-representation is $|G|^{1/2}$.
\end{enumerate}
Moreover, if these conditions are satisfied, then any unitary $\sigma$-representation is unitarily equivalent to a multiple of~$\pi$.
\end{lemma}

Assume now that $G=P\bowtie N$. Denote by $\lambda^G,\lambda^P,\lambda^N$ the left regular representations of $G,P,N$. The left actions of $G$ on $G/P$ and $G/N$ define quasi-regular representations
$$
\lambda^{G/P}\colon G\to U(\ell^2(G/P)),\qquad \lambda^{G/N}\colon G\to U(\ell^2(G/N)).
$$
We identify the sets $G/P$ and $G/N$ with $N$ and $P$, resp. Then obviously
$$
\lambda^{G/P}_n=\lambda^N_n\quad (n\in N),\qquad \lambda^{G/N}_p=\lambda^P_p\quad (p\in P).
$$
On the other hand, the operators
$$
U_p:=\lambda^{G/P}_p\in U(\ell^2(N)),\qquad V_n:=\lambda^{G/N}_n\in U(\ell^2(P)),
$$
are defined by dressing transformations:
$$
(U_pf)(n)=f(p^{-1}\rhd n)\quad (f\in\ell^2(N)),\qquad (V_nf)(p)=f(p\blhd n)\quad (f\in\ell^2(P)).
$$
Since $\lambda^{G/N}$ and $\lambda^{G/P}$ are representations, we have the following relations:
\begin{equation}\label{eq:exchange-op}
\lambda^P_pV_n=V_{p\rhd n}\lambda^P_{p\lhd n},\qquad U_p\lambda^N_n=\lambda^N_{p\rhd n}U_{p\lhd n}.
\end{equation}

Now, fix a bi-$1$-cocycle $\E\colon P\times N\to\T$. We then twist our quasi-regular representations to get two Schr\"odinger-type projective representations
$$
\pi^P\colon G\to U(\ell^2(P)),\qquad \pi^N\colon G\to U(\ell^2(N))
$$
by letting
$$
\pi^P(pn):=\lambda^P_p\overline{\E(\cdot,n)}V_n,\qquad \pi^N(pn):=\E(p,\cdot)U_p\lambda^N_n,
$$
where $\overline{\E(\cdot,n)}$ is viewed as the operator of multiplication by this function, and similarly for $\E(p,\cdot)$.

\begin{lemma}
Both $\pi^P$ and $\pi^N$ are unitary $\sigma_\E$-representations.
\end{lemma}

\bp
If a function $f$ on $P$ is viewed as an operator of multiplication by $f$ on $\ell^2(P)$, then $V_nf=f(\cdot\blhd n)V_n$. From this and the cocycle relations one sees immediately that $\pi^P|_N$ is a homomorphism. In a similar way one sees that $\pi^N|_P$ is a homomorphism. It follows that in order to prove the lemma we only need to check the identity
$$
\E(p,n)\pi(p^{-1})\pi(n)=\pi(p^{-1}\rhd n)\pi(p^{-1}\lhd n)
$$
for $\pi=\pi^P$ and $\pi=\pi^N$, since the property $\pi(pn)=\pi(p)\pi(n)$ and  the same computation as in~\eqref{eq:cocycle-computation} show then that $\pi$ is a $\sigma_\E$-representation.

Consider first the representation $\pi^P$. We have
$$
\E(p,n)\pi^P(p^{-1})\pi^P(n)=\E(p,n)\lambda^P_{p^{-1}}\overline{\E(\cdot,n)}V_n=\E(p,n)\overline{\E(p\,\cdot, n)}\lambda^P_{p^{-1}}V_n,
$$
$$
\pi^P(p^{-1}\rhd n)\pi^P(p^{-1}\lhd n)=\overline{\E(\cdot, p^{-1}\rhd n)}V_{p^{-1}\rhd n}\lambda^P_{p^{-1}\lhd n}.
$$
Since $\lambda^P_{p^{-1}}V_n=V_{p^{-1}\rhd n}\lambda^P_{p^{-1}\lhd n}$ by~\eqref{eq:exchange-op}, we thus need to show that
$$
\E(p,n)\overline{\E(pp', n)}=\overline{\E(p', p^{-1}\rhd n)}.
$$
This is just the first cocycle identity in~\eqref{eq:b1-cocycle1}.

Consider now the representation $\pi^N$. In this case we need to check the identity
$$
\E(p,n)\E(p^{-1}, n')=\E(p^{-1}\lhd n, (p^{-1}\rhd n)^{-1}n').
$$
The right hand side, by the second cocycle identity in~\eqref{eq:b1-cocycle1}, equals
$$
\E(p^{-1}\lhd n, (p^{-1}\rhd n)^{-1})\E\Big(\big((p^{-1}\lhd n)^{-1}\lhd (p^{-1}\rhd n)^{-1}\big)^{-1},n'\Big).
$$
In this product, the first factor equals $\E(p,n)$ by~\eqref{eq:flip}. Since $(p^{-1}\lhd n)^{-1}=p\lhd(p^{-1}\rhd n)$ by~\eqref{eq:dress2}, the first argument in the second factor is $p^{-1}$, so this factor equals $\E(p^{-1},n')$, as needed.
\ep

Our next goal is to construct certain intertwiners. The key player is the Fourier transform-type operator $\F$ with matrix $(\overline{\E(p,n)})_{p\in P,n\in N}$:
\begin{equation}\label{eq:Fourier}
\F\colon\ell^2(N)\to\ell^2(P),\quad (\F f)(p):=\sum_{n\in N}\overline{\E(p,n)}f(n).
\end{equation}

\begin{lemma}\label{lem:Fourier1}
We have $\F\pi^N(g)=\pi^P(g)\F$ for all $g\in G$.
\end{lemma}

\bp
It suffices to check this on the elements of $P$ and $N$. Consider first $g=p\in P$. In this case it is more convenient to check that $\lambda^P_{p^{-1}}\F\pi^N(p)=\F$. We compute:
\begin{align*}
(\lambda^P_{p^{-1}}\F\pi^N(p)f)(p') & =\sum_n\overline{\E(pp',n)}\E(p,n)f(p^{-1}\rhd n) \\
   & =\sum_n\overline{\E(p',p^{-1}\rhd n)}f(p^{-1}\rhd n)=(\F f)(p').
\end{align*}
For $g=n\in N$, the verification of the identity $\F\lambda^N_n=\pi^P(n)\F$ is equally straightforward and we omit it.
\ep

We will denote by one letter $J$ the complex conjugation on various function spaces. The contragredient representation to $\pi^P$ is defined by
$$
\bar\pi^P\colon G\to U(\ell^2(P)),\quad \bar\pi^P(g):=J\pi^P(g)J.
$$
It is clear that this is a $\sigma_{\bar\E}$-representation. In fact, we see from the definitions that this representation is defined in exactly the same way as~$\pi^P$, but with $\E$ replaced by $\bar\E$.

The representation $\pi^P\otimes\bar\pi^P$ is an honest representation of $G$. Our next goal is to construct an intertwiner between this representation and the regular representation. For this we need to introduce a couple of auxiliary operators. Recall first that the \emph{multiplicative unitary} of $P$ is the unitary operator $W_P$ on $\ell^2(P)\otimes\ell^2(P)=\ell^2(P\times P)$ defined by
$$
(W_Pf)(p_1,p_2):=f(p_1,p_1^{-1}p_2).
$$
It has the property
\begin{equation}\label{eq:mult-implementation}
W_P(\lambda^P_p\otimes 1)=(\lambda^P_p\otimes \lambda^P_p)W_P.
\end{equation}
We also define a unitary operator
$$
T\colon\ell^2(G)\to\ell^2(P)\otimes\ell^2(N)\quad\text{by}\quad (Tf)(p,n):=f(pn).
$$

We now define an operator
\begin{equation}\label{eq:K-def}
K\colon \ell^2(G)\to\ell^2(P)\otimes\ell^2(P)\quad\text{by}\quad K:=W_P(1\otimes\bar\F)T,
\end{equation}
where $\bar\F:=J\F J$, or in other words, $\bar\F$ is defined in the same way as $\F$ but using $\bar\E$ instead.
The definition of $K$ is inspired by the Kohn--Nirenberg quantization, cf.~\cite{BGNT5}*{Definition~2.13}.

\begin{lemma}\label{lem:Fourier2}
We have $K\lambda^G_g=(\pi^P(g)\otimes\bar\pi^P(g))K$ for all $g\in G$.
\end{lemma}

\bp
We will check this separately on the elements of $P$ and $N$. Consider first $g=p\in P$. By definition and identity~\eqref{eq:mult-implementation}, we have
$$
(\pi^P(p)\otimes\bar\pi^P(p))K=(\lambda^P_p\otimes\lambda^P_p)W_P(1\otimes\bar\F)T=W_P(1\otimes\bar\F)(\lambda^P_p\otimes1)T,
$$
which proves what we need, as we obviously have $(\lambda^P_p\otimes1)T=T\lambda^G_p$.

Consider now $g=n\in N$. By definition,
\begin{align*}
  \big((\pi^P(n)\otimes\bar\pi^P(n))Kf\big)(p_1,p_2) & =\overline{\E(p_1,n)}\E(p_2,n)\big(W_P(1\otimes\bar\F)Tf\big)(p_1\blhd n,p_2\blhd n) \\
  & =\overline{\E(p_1,n)}\E(p_2,n)\big((1\otimes\bar\F)Tf\big)\big(p_1\blhd n,(p_1\blhd n)^{-1}(p_2\blhd n)\big).
\end{align*}
By \eqref{eq:dress1} and \eqref{eq:dress2}, we have $(p_1^{-1}p_2)\blhd(p_1^{-1}\rhd n)=(p_1\blhd n)^{-1}(p_2\blhd n)$. Therefore we can write the above expression as
$$
\overline{\E(p_1,n)}\E(p_2,n)\big((1\otimes\bar\F)Tf\big)\big(p_1\blhd n,(p_1^{-1}p_2)\blhd(p_1^{-1}\rhd n)\big),
$$
which, in turn, can be written as
$$
\overline{\E(p_1,n)}\big((1\otimes \E(\cdot,n)\lambda^P_{p_1}V_{p_1^{-1}\rhd n}\bar\F)Tf\big)(p_1\blhd n,p_2).
$$
Since $V_{p_1^{-1}\rhd n}=\overline{\E(\cdot,p_1^{-1}\rhd n)}\bar\pi^P(p_1^{-1}\rhd n)$ by definition and $\bar\pi^P(p_1^{-1}\rhd n)\bar\F=\bar\F\lambda^N_{p_1^{-1}\rhd n}$ by Lemma~\ref{lem:Fourier1},  the above expression equals
\begin{equation}\label{eq:K1}
\overline{\E(p_1,n)}\big((1\otimes\E(\cdot,n)\lambda^P_{p_1}\overline{\E(\cdot,p_1^{-1}\rhd n)}\bar\F\lambda^N_{p_1^{-1}\rhd n})Tf\big)(p_1\blhd n,p_2).
\end{equation}
Next, we have
$$
\lambda^P_{p_1}\overline{\E(\cdot,p_1^{-1}\rhd n)}=\overline{\E(p_1^{-1}\cdot,p_1^{-1}\rhd n)}\lambda^P_{p_1}=\overline{\E(p_1^{-1},p_1^{-1}\rhd n)\E(\cdot,n)}\lambda^P_{p_1},
$$
and since $\E(p_1^{-1},p_1^{-1}\rhd n)=\overline{\E(p_1,n)}$, we see that~\eqref{eq:K1} equals
$$
\big((1\otimes\lambda^P_{p_1}\bar\F\lambda^N_{p_1^{-1}\rhd n})Tf\big)(p_1\blhd n,p_2)=\big((1\otimes\bar\F\lambda^N_{p_1^{-1}\rhd n})Tf\big)(p_1\blhd n,p_1^{-1}p_2).
$$
To summarize, we have shown that
\begin{equation}\label{eq:K2}
\big((\pi^P(n)\otimes\bar\pi^P(n))Kf\big)(p_1,p_2)=\big((1\otimes\bar\F\lambda^N_{p_1^{-1}\rhd n})Tf\big)(p_1\blhd n,p_1^{-1}p_2).
\end{equation}

On the other hand, we have
$$
(K\lambda^G_nf)(p_1,p_2)=\big((1\otimes\bar\F)T\lambda^G_nf\big)(p_1,p_1^{-1}p_2).
$$
Comparing this with~\eqref{eq:K2} we see that in order to finish the proof it suffices to show that
$$
\big((1\otimes\lambda^N_{p_1^{-1}\rhd n})Tf\big)(p_1\blhd n,n')=(T\lambda^G_nf)(p_1,n')
$$
for all $n'\in N$, that is,
$$
(p_1\blhd n)(p_1^{-1}\rhd n)^{-1}n'=n^{-1}p_1n',
$$
which is indeed true by the definition of the dressing actions.
\ep

We remark that the above proof did not use the definition of $\F$, only the intertwining property established in Lemma~\ref{lem:Fourier1}.

\smallskip

We are now ready to prove the main result of the paper.

\begin{theorem}\label{thm:main}
Assume $G=P\bowtie N$ and $\E\colon P\times N\to\T$ is a bi-$1$-cocycle. Then the following conditions are equivalent:
\begin{enumerate}
  \item the cocycle $\sigma_\E\in Z^2(G;\T)$ defined by~\eqref{eq:MackeyE} is nondegenerate;
  \item the matrix $(\E(p,n))_{p\in P,n\in N}$ is invertible;
  \item the matrix $(\E(p,n))_{p\in P,n\in N}$ is Hadamard.
\end{enumerate}
\end{theorem}

\bp
Obviously, (3) implies (2), so we only need to prove the implications (1)$\Rightarrow$(3) and (2)$\Rightarrow$(1).

\smallskip

(1)$\Rightarrow$(3) Since $\pi^P$ and $\pi^N$ are unitary $\sigma_\E$-representations of dimension~$|P|$ and~$|N|$, resp., and $|P|\,|N|=|G|$, by Lemma~\ref{lem:Nondeg} it follows that $|P|=|N|=|G|^{1/2}$, both representations are irreducible and we have a unitary equivalence $\pi^P\sim\pi^N$. Then $\F$, being a nonzero intertwiner between irreducible unitary $\sigma_\E$-representations by Lemma~\ref{lem:Fourier1}, must be a scalar multiple of a unitary operator by Schur's lemma. Thus, the matrix $(\E(p,n))_{p,n}$ is Hadamard.

\smallskip

(2)$\Rightarrow$(1) Since the invertibility of $(\E(p,n))_{p,n}$ already implies that $|P|=|N|=|G|^{1/2}$, by Lem\-ma~\ref{lem:Nondeg} it suffices to show that the $\sigma_\E$-representation $\pi^P$ is irreducible. Consider the representation $\Ad\pi^P$ of $G$ on $B(\ell^2(P))$. Then irreducibility of $\pi^P$ is equivalent to the claim that the space of invariant vectors for $\Ad\pi^P$ is one-dimensional. We have a standard equivalence $\pi^P\otimes\bar\pi^P\sim\Ad\pi^P$, explicitly given by mapping $\xi\otimes\zeta\in\ell^2(P)\otimes\ell^2(P)$ into the rank-one operator $(\cdot,J\zeta)\xi\in B(\ell^2(P))$. By Lemma~\ref{lem:Fourier2} and the assumption of invertibility of $\F$ we then conclude that $\Ad\pi^P\sim\lambda^G$. Since the space of vectors fixed by the operators of the regular representation is one-dimensional, this gives the result.
\ep

\begin{remark}
Another, and even shorter, proof of the implication (2)$\Rightarrow$(1) can be obtained as follows. Observe first that the space $\Hom_G(\pi^N,\pi^P)$ is one-dimensional and spanned by~$\F$ for any bi-$1$-cocycle $\E$, no nondegeneracy assumptions are needed. Indeed, assuming that $S\colon\ell^2(N)\to\ell^2(P)$ intertwines~$\pi^N$ and~$\pi^P$, consider the vector $S\delta_e$. Since $\delta_e$ is fixed by $\pi^N(P)$, this vector is fixed by $\lambda^P(P)$, hence it is a scalar multiple $c\,\un$ of the constant function $1$ on $P$. Then
$$
S\delta_n=S\lambda^N_n\delta_e=\pi^P(n)S\delta_e=c\,\pi^P(n)\un=c\,\overline{\E(\cdot,n)}=c\,\F\delta_n,
$$
hence $S=c\,\F$. Now, if $\F$ is invertible, then we conclude that $\pi^N$ and $\pi^P$ are equivalent and $\operatorname{End}_G(\pi^P)$ is one-dimensional, hence $\pi^P$ is irreducible, proving the desired implication. The advantage of the proof using the operator $K$ that we gave is that it generalizes to infinite-dimensional representations. This operator will also allow us to compute dual cocycles associated with $\sigma_\E$.  \ee
\end{remark}

Let us reformulate the theorem in terms of $2$-cocycles instead of bi-$1$-cocycles.

\begin{corollary}\label{cor:cocycle-Hadamard}
Assume $G=P\bowtie N$ and $\sigma\in Z^2(G;\T)$ is a cocycle such that $\sigma|_{P\times P}=1$ and $\sigma|_{N\times N}=1$. Then $\sigma$ is nondegenerate if and only if the matrix
\begin{equation}\label{eq:sigma-matrix}
\big(\sigma(p\rhd n, p\lhd n)\overline{\sigma(p,n)}\big)_{p\in P,n\in N}
\end{equation}
is invertible, if and only if it is Hadamard.
\end{corollary}

\bp
This follows from Theorem~\ref{thm:main} and Proposition~\ref{prop:Mackey}. 
\ep

\begin{remark}
The matrix~\eqref{eq:sigma-matrix} depends on the choice of $\sigma$, while nondegeneracy of $\sigma$ depends only on the cohomology class $[\sigma]$. Therefore invertibility and the Hadamard property of the matrix also depend only on the cohomology class. At the same time, when the matrix is non-invertible, its rank depends on the cocycle, not just on its cohomology class. In order to see this, assume $N$ is normal and abelian, so that $\rhd\colon P\times N\to N$ is an action by automorphisms and $G=P\ltimes N$. Given a character $\chi\in \widehat N$, consider the function $f(pn):=\chi(n)$ and the coboundary $\sigma:=\partial f$. Then
$$
\sigma(p_1n_1,p_2n_2)=\chi(n_1)\overline{\chi(p_2^{-1}\rhd n_1)},\qquad \sigma(p\rhd n, p\lhd n)\overline{\sigma(p,n)}=\chi(p\rhd n)\overline{\chi(n)}.
$$
We see that $\sigma$ has trivial restrictions to $P$ and $N$, as required, but the rank of the matrix~\eqref{eq:sigma-matrix} equals the number of elements in the $P$-orbit of~$\chi$ in~$\widehat N$ and so in general depends on the choice of~$\chi$.\ee
\end{remark}

Theorem~\ref{thm:A} from the introduction provides yet another possible formulation of Theorem~\ref{thm:main}. We spell out the details.

\bp[Proof of Theorem~\ref{thm:A}]
Assume $\pi$ is a projective unitary representation of $G$ of Heisenberg type with phase~$\Phi$. Let $\sigma$ be the cocycle corresponding to~$\pi$. Put $\E(p,n):=\Phi(p^{-1},n)$. Then from Proposition~\ref{prop:Mackey} and identity~\eqref{eq:exchange} we see that $\E$ is a bi-$1$-cocycle and $[\sigma_\E]=[\sigma]$. Since $\dim\pi=|G|^{1/2}$ by assumption, irreducibility of $\pi$ is equivalent to nondegeneracy of $\sigma_\E$. Hence the equivalence of (1)--(3) in Theorem~\ref{thm:A} follows from Theorem~\ref{thm:main}. The uniqueness statement about $\pi$ follows then from simplicity of~$\C^{\sigma}G$.
\ep

Every nondegenerate class in $H^2(G;\T)$ defines a dual $2$-cocycle~~\cites{MR733301,MR990110,MR1264314,MR1747109}, or more precisely, a cohomology class of such cocycles, where two dual $2$-cocycles $\Omega_1$ and $\Omega_2$ are called cohomologous if there exists an invertible element $b\in\C G$ such that
$$
\Omega_1=(b\otimes b)\Omega_2\Delta(b)^{-1}.
$$
Our next goal is to find a representative for the cohomology class of dual $2$-cocycles corresponding to a nondegenerate class $[\sigma_\E]$. We will follow the procedure described in~\cite{BGNT3}*{Section~2.2}, which produces a unitary dual $2$-cocycle.

Assuming that $\E$ satisfies the equivalent conditions of Theorem~\ref{thm:main}, define a linear isomorphism
\begin{equation}\label{eq:op-prime-def}
\Op\colon C(G)\to B(\ell^2(P))\quad\text{by}\quad \Op(f):=|G|^{-1/2}Kf,
\end{equation}
where we identified $B(\ell^2(P))$ with $\ell^2(P)\otimes\ell^2(P)$ as in the proof of Theorem~\ref{thm:main}. The factor~$|G|^{-1/2}$ here is the product of two factors $|G|^{-1/4}$, one of which comes from the normalization of $\F$ to make it unitary, and the other from the equality $\sum_{g\in G}(\Ad\pi^P(g))(S)=(|G|^{1/4})^2\Tr(S)1$, see~\cite{BGNT3}*{Remark 2.11} and note that our current operator $\Op$ corresponds to the operator~$\Op'$ there. Define a new product $\star$ on $C(G)$ by the identity
$$
\Op(f_1\star f_2)=\Op(f_1)\Op(f_2).
$$
Then there is a unique element $\Omega\in \C G\otimes\C G$ such that
\begin{equation}\label{eq:star}
(f_1\star f_2)(g)=(f_1\otimes f_2)\big((g\otimes g)\Omega^*\big),
\end{equation}
where we view the space $C(G)$ as a dual of $\C G$ in the obvious way. The element $\Omega$ is the unitary dual $2$-cocycle we are after.

By construction this dual cocycle is \emph{nondegenerate}, meaning that the corresponding deformed algebra $(C(G),\star)$ is simple. Nondegeneracy has several other equivalent characterizations, see, e.g., \cite{MR2844801}*{Section~1}.

\begin{proposition}\label{prop:dual-computation}
Assume $G=P\bowtie N$ and $\E\colon P\times N\to\T$ is a bi-$1$-cocycle satisfying the equivalent conditions of Theorem~\ref{thm:main}. Then the corresponding unitary dual $2$-cocycle defined by~\eqref{eq:star} is given by
$$
\Omega=|G|^{-1/2}\sum_{p\in P,n\in N}\overline{\E(p^{-1},n^{-1})}\,n\otimes p.
$$
\end{proposition}

\bp
Recalling that $K=W_P(1\otimes\bar\F)T$ and observing that $W_P=\sum_p\delta_p\otimes\lambda^P_p$, we get
\begin{equation}\label{eq:op-prime}
\Op(\delta_{pn})=|G|^{-1/2}(\cdot,J\lambda^P_p\bar\F\delta_n)\delta_p\in B(\ell^2(P)).
\end{equation}
As before, write $\un\in\ell^2(P)$ for the constant function $1$. Then $\bar\F\delta_e=\un$ and we get
$$
\Op(\delta_{p})=|G|^{-1/2}(\cdot,\un)\delta_p.
$$
It follows then that
\begin{equation}\label{eq:op-prime-product}
\Op(\delta_{pn})\Op(\delta_e)=|G|^{-1}(\delta_e,J\lambda^P_p\bar\F\delta_n)(\cdot,\un)\delta_p=|G|^{-1/2}\E(p^{-1},n)\Op(\delta_p).
\end{equation}
Therefore the new product $\star$ on $C(G)$ satisfies
$\delta_{pn}\star \delta_e=|G|^{-1/2}\E(p^{-1},n)\delta_p$. Since the product is left $G$-equivariant, we then get
$$
\delta_{hpn}\star \delta_h=|G|^{-1/2}\E(p^{-1},n)\delta_{hp}
$$
for all $h\in G$, $p\in P$ and $n\in N$. In other words, for all $g,h\in G$,
\begin{equation}\label{eq:star2}
\delta_{g}\star \delta_h=|G|^{-1/2}\E(p^{-1},n)\delta_{hp},\quad\text{where}\quad pn=h^{-1}g.
\end{equation}

From the defining relation~\eqref{eq:star} we deduce then that
$$
\Omega^*=\sum_{g,h\in G}(\delta_g\star\delta_h)(e)\,g\otimes h=|G|^{-1/2}\sum_{p\in P,n\in N}\E(p,n)\,n\otimes p.
$$
By taking the adjoints we get the result.
\ep

\begin{corollary}\label{cor:another-dual}
In the setting of the proposition, the element
$$
|G|^{-1/2}\sum_{p\in P,n\in N}\E(p^{-1},n^{-1})\,p\otimes n
$$
is a unitary dual $2$-cocycle cohomologous to $\Omega$.
\end{corollary}

\bp
By applying the inverse to the defining relation~\eqref{eq:exchange} we see that the function $\E'(n,p):=\E(p,n)^{-1}$ is a bi-$1$-cocycle for the matched pair $(N,P)$ and $[\sigma_{\E'}]=[\sigma_\E]$. Hence the corollary follows from Proposition~\ref{prop:dual-computation} applied to the matched pairs $(P,N)$ and $(N,P)$ and the bi-$1$-cocycles~$\E$ and~$\E'$.
\ep

We see that the cocycles we got are of the form studied in Proposition~\ref{prop:dual-bi1}. We can now return to the question of invertibility left open in that proposition and finish the proof of Theorem~\ref{thm:B} from the introduction. Since we already know by Corollary~\ref{cor:roots} that any $\C^\times$-valued bi-$1$-cocycle is automatically $\T$-valued, all that remains is to prove the ``only if'' part of the following theorem.

\begin{theorem}\label{thm:dual}
Assume $G=P\bowtie N$ and $\E\colon P\times N\to\T$ is a bi-$1$-cocycle. Then the element
$$
\Omega:=|G|^{-1/2}\sum_{p\in P,n\in N}\overline{\E(p^{-1},n^{-1})}\,n\otimes p\in \C G\otimes\C G
$$
is invertible if and only if $\E$  satisfies the equivalent conditions of Theorem~\ref{thm:main}.
\end{theorem}

\bp
The ``if'' part follows from Proposition~\ref{prop:dual-computation}. For the opposite implication, assume $\Omega$ is invertible, so that $\Omega$ is a dual $2$-cocycle by Proposition~\ref{prop:dual-bi1} and the fact that the flip map on $\C G\otimes\C G$ preserves the set of dual $2$-cocycles. We want to conclude that the matrix $(\E(p,n))_{p,n}$ is invertible. By passing to the pair $(N,P)$ if necessary we may assume that $|N|\le|P|$.

Define a new product $\star$ on $C(G)$ by~\eqref{eq:star}. Since it is left $G$-equivariant, it is completely specified by the values $(\delta_g\star\delta_h)(e)$. By the computations in the second part of the proof of Proposition~\ref{prop:dual-computation} it follows that this product is given by~\eqref{eq:star2}. Consider also the linear map $\Op\colon C(G)\to B(\ell^2(P))$ defined by~\eqref{eq:op-prime}. The computations in the first part of the proof of Proposition~\ref{prop:dual-computation} did not use any nondegeneracy assumptions, so by~\eqref{eq:op-prime-product} we have
$$
\Op(\delta_{pn})\Op(\delta_e)=|G|^{-1/2}\E(p^{-1},n)\Op(\delta_p)=\Op(\delta_{pn}\star\delta_e).
$$
Since $\Op$ intertwines $\lambda^G$ with $\Ad\pi^P$ by Lemma~\ref{lem:Fourier2},  by applying $\Ad\pi^P(h)$ we conclude that
$$
\Op(\delta_{g})\Op(\delta_h)=\Op(\delta_g\star\delta_h)
$$
for all $g,h\in G$, that is, $\Op$ is a representation of the algebra $(C(G),\star)$.

By \cite{MR1264314}*{Proposition~7}, the invertibility of $\Omega^*$ implies that the algebra $(C(G),\star)$ is semisimple. It follows that its image under $\Op$ is semisimple as well. From~\eqref{eq:op-prime} and the fact that $\lambda^P_p\F=\F\pi^N(p)$ by Lemma~\ref{lem:Fourier1} we see that $\Op(C(G))$ is spanned by the operators of the form $(\cdot,\zeta)\xi$ with $\xi\in\ell^2(P)$ and $\zeta\in\operatorname{Im}\F$. Thus, it consists of all operators on $\ell^2(P)$ that vanish on $(\operatorname{Im}\F)^\perp$. The operators $S$ such that $(\operatorname{Im}\F)^\perp\subset\ker S$ and $\operatorname{Im}S\subset (\operatorname{Im}\F)^\perp$ form an ideal with zero product in this algebra, so $\Op(C(G))$ is semisimple only when $(\operatorname{Im}\F)^\perp=0$, that is, when $\F$ is surjective. Since $|N|\le|P|$ by assumption, this means that $(\E(p,n))_{p,n}$ must be invertible.
\ep

\bigskip

\section{A Lazard-type correspondence between symplectic and nondegenerate cocycles}

Consider a finite field $\FF_q$ of characteristic $p$. Assume $\g$ is a finite-dimensional nilpotent Lie algebra over $\FF_q$ of nilpotency class $c(\g)$, so that any iterated commutator of length $c(\g)+1$ in~$\g$ is zero. If $c(\g)<p$, then by the classical Lazard correspondence (see, e.g., \cite{MR1615819}) we get a finite nilpotent $p$-group $G$ out of $\g$. Concretely, we take $G:=\g$ as a set, but in order to distinguish between elements of $\g$ and $G$ we write $\exp$ for the identity map $\g\to G$.  The product on $G$ is then given by
$$
\exp (x)\exp (y):=\exp(\BCH(x,y)),
$$
where $\BCH(X,Y)\in\Q\llangle X,Y\rrangle$ is the Baker--Campbell--Hausdorff series in noncommuting variables $X$ and $Y$. The assumption $c(\g)<p$ and explicit formulas for the $\BCH$-series, such as Dynkin's formula
$$
\BCH(X,Y)
=\sum_{n=1}^{\infty}\frac{(-1)^{n-1}}{n}
\sum_{\substack{r_1,s_1,\ldots,r_n,s_n\geq 0\\
                 r_i+s_i>0\ (1\leq i\leq n)}}
\frac{
  [X^{r_1}Y^{s_1}\cdots X^{r_n}Y^{s_n}]
}{
  \left(\sum_{i=1}^{n}(r_i+s_i)\right)
  \prod_{i=1}^{n}r_i!\,s_i!
},
$$
where $[X]:=X$, $[Y]:=Y$ and for a word $Z_1\cdots Z_m$ with $Z_i\in\{X,Y\}$ and $m\ge2$ we put
$$
[Z_1\cdots Z_m]
:=
[Z_1,[Z_2,\ldots,[Z_{m-1},Z_m]\ldots]],
$$
guarantee that the product is well-defined.

We have an adjoint action $\Ad\colon G\to\Aut(\g)$ defined by
$$
\Ad_{\exp x}(y):=e^{\ad_x}y,
$$
and correspondingly a coadjoint action $\Ad^*$ on the dual space $\g^*:=\Hom_{\FF_q}(\g,\FF_q)$. Here again the condition $c(\g)<p$ ensures that the operator $e^{\ad_x}$ is well-defined. Moreover, then $\log e^{\ad_x}=\ad_x$, and as a consequence, given $x\in\g$ and $f\in\g^*$, we have
\begin{equation}\label{eq:ad}
\Ad^*_{\exp x}(f)=f\quad\text{if and only if}\quad f\circ\ad_x=0.
\end{equation}

Assume next that we are given a $2$-cocycle $\omega\in Z^2(\g;\FF_q)$, so that $\omega\colon\g\times\g\to\FF_q$ is an alternating form satisfying
$$
\omega([x,y],z)+\omega([z,x],y)+\omega([y,z],x)=0.
$$
Consider the corresponding central extension $\g_\omega$ of $\g$. As usual, it will be convenient to identify~$\g_\omega$ with the space $\FF_q\oplus\g$ equipped with the Lie bracket
$$
[(a,x),(b,y)]:=(\omega(x,y),[x,y]).
$$
This Lie algebra has nilpotency class $c(\g)$ or $c(\g)+1$. Hence, if $c(\g)+1<p$, by Lazard's correspondence we get a nilpotent group $G_\omega$ and a central extension
$1\to\FF_q\to G_\omega\to G\to1$. We denote by $l_\omega\in Z^2(G;\FF_q)$ the corresponding $2$-cocycle. Explicitly, it is defined by the identity
\begin{equation*}\label{eq:L-cocycle}
\BCH((0,x),(0,y))=(l_\omega(\exp x,\exp y),\BCH(x,y)),
\end{equation*}
where the value of the $\BCH$-series on the left is computed in $\g_\omega$, while on the right in $\g$. Thus,
\begin{equation}\label{eq:L-cocycle2}
l_\omega(\exp x,\exp y)=\sum_{n=1}^{c(\g)+1}\frac{(-1)^{n-1}}{n}
\sum_{\substack{r_1,s_1,\ldots,r_n,s_n\geq 0\\
                 r_i+s_i>0\ (1\leq i\leq n)\\
                 2\le r_1+s_1+\ldots+r_n+s_n\le c(\g)+1}}
\frac{
  [x^{r_1}y^{s_1}\cdots x^{r_n}y^{s_n}]_\omega
}{
  \left(\sum_{i=1}^{n}(r_i+s_i)\right)
  \prod_{i=1}^{n}r_i!\,s_i!
},
\end{equation}
where for a word $z_1\cdots z_m$ with $z_i\in\{x,y\}$ and $m\ge2$ we put
$$
[z_1\cdots z_m]_\omega
:=
\omega(z_1,[z_2,\ldots,[z_{m-1},z_m]\ldots]).
$$

Take a nontrivial character $\chi\colon\FF_q\to\T$, and recall that it has the form $e^{\frac{2\pi i}{p}\tr(c\,\cdot)}$, where $\tr\colon\FF_q\to\FF_p=\Z/p\Z$ is the trace map and $c\in\FF_q^\times$. We then get a cocycle $\chi\circ l_\omega\in Z^2(G;\T)$.

\begin{theorem}
Let $\g$ be a finite-dimensional nilpotent Lie algebra over a finite field $\FF_q$ of characteristic $p$. Assume its nilpotency class satisfies $c(\g)<p-1$ and let $G$ be the corresponding finite nilpotent group. For a cocycle $\omega\in Z^2(\g;\FF_q)$ and a nontrivial character $\chi\colon\FF_q\to\T$, consider the cocycle $\chi\circ l_\omega\in Z^2(G;\T)$, where $l_\omega$ is given by~\eqref{eq:L-cocycle2}. Then the cocycle $\chi\circ l_\omega$ is nondegenerate if and only if the form $\omega$ is symplectic.
\end{theorem}

\bp
Consider the element $f\in\g_\omega^*$ defined by $f(a,x):=a$. By a version of Kirillov's orbit method we have a representation of $G_\omega$ corresponding to the character $\chi\circ f\colon\g_\omega\to\T$. The precise result sufficient for our purposes is given in \cite{MR2403615}*{Theorem~2.6}, according to which there is an irreducible unitary representation $\pi$ of $G_\omega$ with character
$$
\chi_\pi(\exp(a,x))=\frac{1}{|\OO|^{1/2}}\sum_{\tilde\chi\in\OO}\tilde\chi(a,x),
$$
where $\OO$ is the orbit of $\chi\circ f$ in the Pontryagin dual $\widehat\g_\omega:=\Hom(\g_\omega,\T)$ under the action of~$G_\omega$ on~$\widehat\g_\omega$ induced by the adjoint action on $\g_\omega$. In particular, $\dim\pi=|\OO|^{1/2}$.

Since $\FF_q\oplus0\subset\g_\omega$ is central, it is clear that $\tilde\chi(a,0)=\chi(a)$ for all $\tilde\chi\in\OO$ and $a\in\FF_q$. Therefore if we view $\FF_q$ as a subgroup of $G_\omega$ by identifying $a$ with $\exp(a,0)$, then $\chi_\pi(a)=|\OO|^{1/2}\chi(a)$, and hence $\pi(a)=\chi(a)1$. This shows that if we view $\pi$ as a projective representation of $G$, then the corresponding cocycle is $\chi\circ l_\omega$. It follows that this cocycle is nondegenerate if and only if $\dim\pi=|G|^{1/2}$, that is, if and only if $|\OO|=|G|$. Since $|G_\omega|=q|G|$, the last condition is equivalent to saying that the stabilizer of $\chi\circ f$ in $G_\omega$ has order~$q$.

Since the composition of $\chi$ with any nonzero element of $\g_\omega^*$ gives a nontrivial character of~$\g_\omega$, the stabilizer of $\chi\circ f$ coincides with the stabilizer of $f$ under the coadjoint action. By~\eqref{eq:ad}, the latter stabilizer is $\exp(\FF_q\oplus\operatorname{rad}(\omega))$, so it has order~$q$ if and only if $\operatorname{rad}(\omega)=0$.
\ep

\begin{corollary}\label{cor:Lagrangian}
In the setting of the theorem, assume we are given a symplectic cocycle $\omega\in Z^2(\g;\FF_q)$. Assume $\p,\n\subset\g$ are complementary Lagrangian subalgebras with respect to~$\omega$, and consider the corresponding subgroups $P,N\subset G$. Then $G=P\bowtie N$ and the matrix
\begin{equation}\label{eq:Hadamard}
\big(\chi\big(l_\omega(p\rhd n, p\lhd n)-l_\omega(p,n)\big)\big)_{p\in P,n\in N}
\end{equation}
is Hadamard.
\end{corollary}

\bp
Since $\p\cap\n=0$, we have $P\cap N=\{e\}$, and as $|P|\,|N|=|G|$, it follows that $G=P\bowtie N$. Consider the cocycle $\sigma:=\chi\circ l_\omega$ on $G$. By construction it is clear that $\sigma|_{P\times P}=1$ and $\sigma|_{N\times N}=1$. Since $\sigma$ is nondegenerate by the theorem, the result follows from Corollary~\ref{cor:cocycle-Hadamard}.
\ep

The matrix~\eqref{eq:Hadamard} can be viewed as a quantization of the Fourier matrix $\big(\chi(-\omega(x,y))\big)_{x\in\p,y\in\n}$. Namely, in the setting of the above corollary, consider the Lie algebra $\g[t]$ over the polynomial ring $\FF_q[t]$ with Lie bracket
$$
[xt^n,yt^m]:=[x,y]t^{n+m+1}\qquad (x,y\in\g).
$$
The Lazard correspondence still applies and we get an infinite group $G_t$ together with a bijection $\exp\colon\g[t]\to G_t$. Extend $\omega$ by $\FF_q[t]$-bilinearity to $\g[t]$. It remains a Lie algebra $2$-cocycle for the Lie bracket on $\g[t]$. We then get a central extension $1\to\FF_q[t]\to G_{t,\omega}\to G_t\to1$ and the corresponding cocycle $l_{t,\omega}\in Z^2(G_t;\FF_q[t])$. From the BCH-series and the definition of the bracket on $\g[t]$ it is clear that
$$
l_{t,\omega}(\exp x,\exp y)=\frac{1}{2}\omega(x,y)+O(t)\qquad (x,y\in\g[t]).
$$

Consider the subgroups $P_t,N_t\subset G_t$ corresponding to $\p[t],\n[t]\subset\g[t]$. Clearly, $P_t\cap N_t=\{e\}$. We claim that $(P_t,N_t)$ is a matched pair and the inverse of the product map $P_t\times N_t\to G_t$ is polynomial over~$\FF_q[t]$ in the exponential coordinates. The Lagrangian property does not play any role here, so we will prove this claim for any pair of complementary subalgebras $\p,\n\subset\g$. The proof is by induction on $\dim\g$. When $\dim \g=1$, there is nothing to prove. For the induction step, assume $\p\ne\g$, as otherwise again there is nothing to prove. Since $\g$ is nilpotent, there is a codimension one ideal $\mathfrak h\subset\g$ containing $\p$. Let $H_t:=\exp(\mathfrak h[t])$. Choose an element $u\in\n\setminus\mathfrak h$. Let $f\in\g^*$ be the linear functional such that $\mathfrak h=\ker f$ and $f(u)=1$. Extend~$f$ by $\FF_q[t]$-linearity to~$\g[t]$.

Take any $z\in\g[t]$. Then $\exp(z)\exp(-f(z)u)\in H_t$, since $\g[t]/\mathfrak h[t]\cong\FF_q[t]$ is an abelian Lie algebra and hence $G_t/H_t$ is isomorphic to $\FF_q[t]$, with an isomorphism explicitly given by $\exp(z)H_t\mapsto f(z)$. By the inductive assumption applied to $\mathfrak h$ we can write $\exp(z)\exp(-f(z)u)$ as $\exp(x)\exp(y)$ for uniquely defined $x\in\p[t]$ and $y\in(\n\cap\mathfrak h)[t]$, and both $x$ and $y$ are given by polynomial expressions over $\FF_q[t]$ in $\exp^{-1}\big(\exp(z)\exp(-f(z)u)\big)$. Then
$$
\exp(z)=\exp(x)\exp(y)\exp(f(z)u).
$$
Since $\exp(y)\exp(f(z)u)\in N_t$, this finishes the proof of the induction step and hence of our claim. Note also that since the Lie bracket on $\g[t]$ has zero constant term, the corresponding dressing actions $\rhd_t\colon P_t\times N_t\to N_t$ and $\lhd_t\colon P_t\times N_t\to P_t$ are trivial modulo $t$ in the exponential coordinates.

We can now define a polynomial map $\Psi\colon\p\times\n\to\FF_q[t]$ by
$$
\Psi(x,y):=l_{t,\omega}(\exp(x)\rhd_t\exp(y), \exp(x)\lhd_t\exp(y))-l_{t,\omega}(\exp(x),\exp(y)).
$$
By construction we have $\Psi(x,y)|_{t=0}=-\omega(x,y)$ and
$$
\Psi(x,y)|_{t=1}= l_\omega(\exp(x)\rhd\exp(y), \exp(x)\lhd\exp(y))-l_\omega(\exp(x),\exp(y)).
$$
More generally, we can specialize $\Psi$ at any $s\in\FF_q$. The matrix $(\chi(\Psi(x,y)|_{t=s}))_{x\in\p,y\in\n}$ is Hadamard by Corollary~\ref{cor:Lagrangian} applied to $\g$ equipped with the scaled Lie bracket $s[\cdot,\cdot]$. Thus we have a polynomial family of phases producing Hadamard matrices, with the Fourier matrix $\big(\chi(-\omega(x,y))\big)_{x\in\p,y\in\n}$ at $s=0$ and the matrix~\eqref{eq:Hadamard} at $s=1$. This justifies the view that \eqref{eq:Hadamard} is a quantization of the Fourier matrix.

\begin{example}
Let $q\geq 7$ be a prime, and consider the following eight-dimensional Lie algebra~$\g$ over~$\mathbb F_q$ with basis
$a_1,a_2,a_3,a_4,b_1,b_2,b_3,b_4$. The nonzero brackets, up to skew-symmetry, are
\begin{align*}
[a_1,a_4]&=[a_1,b_4]=a_2,\\
[b_1,a_4]&=[b_1,b_4]=b_2,\\
[a_1,a_2]&=a_3,\\
[b_1,b_2]&=b_3,\\
[a_1,b_2]&=[b_1,a_2]=a_3+b_3,\\
[a_1,b_3]&=[b_1,a_3]=a_4-b_4.
\end{align*}
This is a finite field version of the complexification of a real Lie algebra from \cite{MR4868109}*{Theorem~5.1(ii)}. The lower central series of $\g$ is
$$
\g\supset\operatorname{span}\{a_2,b_2,a_3,b_3,a_4-b_4\}\supset
\operatorname{span}\{a_3,b_3,a_4-b_4\}\supset\mathbb F_q(a_4-b_4)\supset0.
$$
Therefore $c(\g)=4$. Writing $a_i^*,b_i^*$ for the dual basis vectors, define
$$
\omega:
=
a_1^*\wedge b_1^*
-a_2^*\wedge b_3^*
+a_3^*\wedge b_2^*
-a_4^*\wedge b_4^*.
$$
This is a symplectic $2$-cocycle. The Lie subalgebras
$$
\p:=\operatorname{span}
\{a_1,a_2,a_3,a_4\}
\quad\text{and}\quad
\n:=\operatorname{span}
\{b_1,b_2,b_3,b_4\}
$$
are complementary Lagrangian subalgebras.
Thus, we are in a position to apply Corollary~\ref{cor:Lagrangian}.

In order to do computations it will be convenient to slightly modify the exponential coordinates. Observe that if we put
$$
\p_0:=\operatorname{span}\{a_2,a_3,a_4\},
\qquad
\n_0:=\operatorname{span}\{b_2,b_3,b_4\},
$$
then $\mathfrak a:=\p_0\oplus\n_0$ is an abelian ideal of $\g$.
In view of this and the relation $[a_1,b_1]=0$ we parameterize~$P$ and~$N$~by
\begin{align*}
p(x,r,s,t)&:=\exp(r a_2+s a_3+t a_4)\exp(xa_1),\\
n(y,R,S,T)&:=\exp(yb_1)\exp(Rb_2+Sb_3+Tb_4).
\end{align*}
In these coordinates, the group laws are
$$
(x,r,s,t)(x',r',s',t')\\
=
\Big(
x+x',
r+r'+xt',
s+s'+xr'+\frac{x^2t'}2,
t+t'
\Big)
$$
on $P$, and
$$
(y,R,S,T)(y',R',S',T')
=
\Big(
y+y',
R+R'-y'T,
S+S'-y'R+\frac{(y')^2T}{2},
T+T'
\Big)
$$
on $N$. To compute the dressing actions, put
$$
U_p:=r a_2+s a_3+t a_4,\qquad
V_n:=Rb_2+Sb_3+Tb_4,
$$
and consider the operators
\[
D_p:=x\ad_{a_1}|_{\mathfrak a},\qquad D_n:=y\ad_{b_1}|_{\mathfrak a}.
\]
Since $\mathfrak a$ is abelian and $[a_1,b_1]=0$, we have
\[
pn=
\exp(yb_1)
\exp\bigl(e^{-D_n}U_p+e^{D_p}V_n\bigr)
\exp(xa_1).
\]
The two terms in the middle exponent are
\begin{align*}
e^{-D_n}U_p
={}&r a_2+(s-yr)a_3
+\Big(t-ys+\frac{y^2r}{2}\Big)a_4\\
&-yt\,b_2
+\Big(-yr+\frac{y^2t}{2}\Big)b_3
+\Big(ys-\frac{y^2r}{2}\Big)b_4,
\\[1ex]
e^{D_p}V_n
={}&xT\,a_2
+\Big(xR+\frac{x^2T}{2}\Big)a_3
+\Big(xS+\frac{x^2R}{2}\Big)a_4\\
&+R b_2+(S+xR)b_3
+\Big(T-xS-\frac{x^2R}{2}\Big)b_4.
\end{align*}
Splitting their sum as $U'_{pn}+V'_{pn}$, with $U'_{pn}\in\p_0$ and $V'_{pn}\in\n_0$, gives
\begin{align*}
p\rhd n
&=
\Big(
y,\,
R-yt,\,
S+xR-yr+\frac{y^2t}{2},
T-xS-\frac{x^2R}{2}+ys-\frac{y^2r}{2}
\Big),\\
p\lhd n
&=
\Big(
x,\,
r+xT,\,
s-yr+xR+\frac{x^2T}{2},
t-ys+\frac{y^2r}{2}+xS+\frac{x^2R}{2}
\Big).
\end{align*}

We next compute the phase
$$
\Psi(p,n):=l_\omega(p\rhd n, p\lhd n)-l_\omega(p,n)\in\FF_q.
$$
Note that if we denote by $z$ the central element $(1,0)\in\g_\omega$, then, in view of~\eqref{eq:exchange} and the facts that $\p$ and $\n$ are Lie subalgebras of $\g_\omega$, so that $P$ and $N$ can be viewed as subgroups of $G_\omega$, the phase is determined by the following relation in $G_\omega$:
\begin{equation}\label{eq:phase}
\exp(\Psi(p,n)z)pn=(p\rhd n)(p\lhd n).
\end{equation}
The $\BCH$-formula, together with the facts that the central extension of $\mathfrak a$ is two-step nilpotent and $[a_1,b_1]_{\g_\omega}=z$, gives then that
$$
\Psi(p,n)
=-xy-\frac12\omega\big(e^{-D_n}U_p,e^{D_p}V_n\big)+\frac12\omega(V'_{pn},U'_{pn}).
$$
The final expression is
\begin{align*}
\Psi(x,r,s,t;y,R,S,T)
=&-xy+r(S+xR)-(s-yr)R\\
&+\Big(t-ys+\frac{y^2r}{2}\Big)
\Big(T-xS-\frac{x^2R}{2}\Big)\\
&-\frac12\Big(
yr^2-2yst+y^2(rt+s^2)
-y^3rs+\frac{y^4r^2}{4}
\Big)\\
&-\frac12\Big(
xR^2-2xST+x^2(S^2-RT)
+x^3RS+\frac{x^4R^2}{4}
\Big).
\end{align*}
Therefore the conclusion is that the following $q^4\times q^4$ matrix is Hadamard for all primes $q\ge7$:
$$
\Big(
e^{\frac{2\pi i}{q}\Psi(x,r,s,t;y,R,S,T)}\Big)_{(x,r,s,t),(y,R,S,T)\in\mathbb F_q^4}.
$$
A direct computation shows that the same is in fact true for all odd primes.

\smallskip

We claim that, for $q\ge5$, the matrix $(e^{\frac{2\pi i}{q}\Psi(u,v)})_{u,v\in\mathbb F_q^4}$
is neither equivalent to a Fourier matrix nor of Di\c{t}\u{a} type. Recall that an Hadamard matrix is of \emph{Di\c{t}\u{a} type}
if it is equivalent, by permutations and multiplication of rows
and columns by scalars of modulus one, to a block matrix
\begin{equation}\label{eq:Dita}
D=\begin{pmatrix}
a_{11}B_1 & \cdots & a_{1k}B_k\\
\vdots & & \vdots\\
a_{k1}B_1 & \cdots & a_{kk}B_k
\end{pmatrix},
\end{equation}
where $A=(a_{ij})$ is an Hadamard matrix of size $k>1$
and $B_1,\ldots,B_k$ are Hadamard matrices of a common
size $m>1$. This construction generalizes the tensor product $A\otimes B$,
obtained when all $B_j$ coincide.

In order to prove our claim we introduce two invariants similar to~\cite{MR2372303}*{Definition~3.1}. Given an Hadamard matrix $H=(H_{ij})_{i,j}$, denote by $m_1(H)$ the maximal number of distinct ordered pairs $(i,i')$ ($i\ne i'$) such that the row quotients $(H_{ij}/H_{i'j})_j$ are equal up to phase factors. Denote by $m_2(H)$ the maximal number of disjoint ordered pairs $(i,i')$ ($i\ne i'$) with the same property.

Let us compute these invariants for our matrix $H:=(e^{\frac{2\pi i}{q}\Psi(u,v)})_{u,v\in\mathbb F_q^4}$. Since $H_{u,0}=1$, proportional row quotients of $H$ are in fact equal. Comparing coefficients in $\Psi(u,v)-\Psi(u',v)$ it is not too difficult to see that this difference determines $(u,u')$ uniquely unless~$u$ and~$u'$ have the same first three coordinates, namely, one distinguishes successively the cases (a) $x\ne x'$, (b)~$x=x'$ but $r\ne r'$, and finally (c) $(x,r)=(x',r')$ but $s\ne s'$.
Coefficient comparison is valid, since the degree in each column variable is less than $q\ge5$. The remaining quotients to consider~are
$$
 \frac{H_{(x,r,s,t+d),v}}{H_{(x,r,s,t),v}}
 =
 \exp\!\Big(\frac{2\pi i d}{q}
 \Big(T-xS-\frac{x^2R}{2}+ys-\frac{y^2r}{2}\Big)\Big),
 \qquad d\ne0.
$$
Such a row quotient determines $(x,r,s,d)$, so the same quotient occurs exactly $q$ times, with all corresponding pairs of rows contained in the same set
of $q$ rows. It follows that
$$
m_1(H)=q,\qquad m_2(H)=\frac{q-1}{2}.
$$

On the other hand, for the Fourier matrix $F:=(\chi(x))_{x\in B,\chi\in\widehat B}$ of any finite abelian group $B$ we have $m_1(F)=|B|$, while
for any Di\c{t}\u{a}-type matrix~\eqref{eq:Dita} of size $km$ we have $m_2(D)\ge m$, since for any fixed $1\le i<i'\le k$ we have $m$ disjoint row pairs $((i,\alpha),(i',\alpha))$, $1\le\alpha\le m$, with the same entrywise quotient. Therefore if our matrix was equivalent to a Fourier matrix or to a matrix of Di\c{t}\u{a} type, we would have $m_1(H)=q^4$ or $m_2(H)\ge q$, resp. This proves our claim. In fact, the claim remains true for $q=3$, but this case requires a separate argument, which we omit.

\smallskip

Let us finally say a few words about $\g$ and $G$, assuming that $q\ge7$. This example is designed to satisfy a couple of special requirements.

Consider the $2$-cocycle $\sigma:=e^{\frac{2\pi i}{q}l_\omega}$. Recall that an isotropic (with respect to the nondegenerate class $[\sigma]$)  subgroup $H\subset G$ is called \emph{Lagrangian} if $|H|=|G|^{1/2}$. By \cite{MR3210715}*{Lemma~2.2}, a normal isotropic subgroup is necessarily abelian. We claim that our group $G$ does not have any normal Lagrangian subgroups. In fact, the order of any normal isotropic subgroup of $G$ is at most $q^3$. Indeed, if $H$ is such a subgroup, then $H=\exp\mathfrak h$ for some abelian ideal $\mathfrak h\subset\g$. Then $\exp\colon\mathfrak h\to H$ is an isomorphism and $\exp^*(l_\omega)=\frac{1}{2}\omega$ on $\mathfrak h$. Since $\operatorname{res}^G_H[\sigma]=0$ by assumption, it follows that the ideal $\mathfrak h\subset\g$ is isotropic with respect to $\omega$. Observe next that we must have $\mathfrak h\subset\mathfrak a$. In order to see this, consider an element
$$
X=ua_1+vb_1+r\in\mathfrak h,\quad\text{with}\quad r\in\mathfrak a.
$$
Then
$$
0=[X,[X,a_4]] =u(u+2v)a_3+v(v+2u)b_3.
$$
As $q\ne 3$, the equations $u(u+2v)=v(v+2u)=0$ force $u=v=0$. Thus, $\mathfrak h\subset\mathfrak a$. Since the form~$\omega|_{\mathfrak a}$ is nondegenerate, we conclude that $\dim \mathfrak h\le\frac{1}{2}\dim\mathfrak a=3$, hence $|H|\le q^3$, proving our claim. By~\cite{MR2534251}*{Theorem~A} it follows that the nondegenerate class $[\sigma]$ cannot be obtained by the bijective $1$-cocycle method. Note also that \cite{MR3210715}*{Theorem~1.9} (see also~\cite{MR3499032}*{Theorem~8.4}) implies that eight is the minimal dimension where one can hope that there are no Lagrangian ideals, so dimension-wise our $\g$ is the smallest possible nilpotent Lie algebra with this property.

The second special property is that the Lagrangian subalgebras $\p$ and $\n$ have trivial intersections with the center $\mathfrak z(\g)=\FF_q(a_4-b_4)$. Thus neither chosen Lagrangian subalgebra contains a nonzero central ideal along which one could perform symplectic reduction. \ee
\end{example}

\bigskip

\end{document}